\documentclass{rspublic}
\usepackage{amsmath,amssymb}
\usepackage{graphicx}
\newcommand*{\mA}{{\mathbb A}}
\newcommand*{\mK}{{\mathbb K}}
\newcommand*{\mR}{{\mathbb R}}
\newcommand*{\mS}{{\mathbb S}}
\newcommand*{\mB}{{\mathbb B}}

\newcommand*{\cK}{{\mathcal K}}
\newcommand*{\pA}{{\mathbb A^{+}}}
\newcommand*{\mO}{{\mathbb O}}
\newcommand*{\mV}{{\mathbb V}}
\newcommand*{\mJ}{{\mathbb J}}
\newcommand*{\mE}{{\mathbb E}}

\newcommand*{\ml}{{1}}
\newcommand*{\la}{{\langle}}
\newcommand*{\ra}{{\rangle}}
\newcommand*{\id}{{\mathrm{id}}}

\newcommand*{\sh}{{\,\llcorner\!\llcorner\!\!\!\lrcorner\,}}
\newcommand*{\nn}{{\mathrm{n}}}
\newcommand*{\eps}{{\epsilon}}
\theoremstyle{plain}
\newtheorem{definition}{Definition}
\theoremstyle{remark}
\newtheorem{remark}{Remark}
\numberwithin{definition}{section}
\numberwithin{remark}{section}
\begin{document}
\title{Stochastic expansions and Hopf algebras}
\author[S.J.A. Malham and A. Wiese]{Simon J.A. Malham and Anke Wiese}
\affiliation{Maxwell Institute for Mathematical Sciences 
              and School of Mathematical and Computer Sciences,
              Heriot-Watt University, Edinburgh EH14 4AS, UK \\
\textsf{Dedicated to Mrs Jeanne Anne Malham} \\
(17th April 2009)}

\label{firstpage}
\maketitle

\begin{abstract}{Stochastic expansions, Hopf algebras, numerical approximation}
We study solutions to nonlinear stochastic differential systems 
driven by a multi-dimensional Wiener process. A useful algorithm 
for strongly simulating such stochastic systems is the 
Castell--Gaines method, which is based on the exponential Lie series. 
When the diffusion vector fields commute,
it has been proved that at low orders this method 
is more accurate in the mean-square error than corresponding stochastic Taylor methods. 
However it has also been shown that when the diffusion vector fields do not commute,
this is not true for strong order one methods. Here we prove that
when there is \emph{no} drift, and the diffusion vector fields
do not commute, the exponential Lie series is usurped by the 
sinh-log series. In other words, the mean-square error associated
with a numerical method based on the sinh-log series, is always
smaller than the corresponding stochastic Taylor error, in fact to 
\emph{all} orders. Our proof utilizes the underlying Hopf algebra 
structure of these series, and a two-alphabet associative algebra
of shuffle and concatenation operations. We illustrate the benefits
of the proposed series in numerical studies.
\end{abstract}

\section{Introduction}
\label{intro}
We are interested in designing strong series solutions
of nonlinear Stratonovich stochastic differential
systems of the form
\begin{equation*}
y_t=y_0+\sum_{i=1}^d\int_0^t V_i(y_\tau)\,\mathrm{d}W_\tau^i.
\end{equation*}
Here $(W^1,\ldots,W^d)$ is a $d$-dimensional Wiener process and
$y_t\in\mathbb R^N$ for some $N\in\mathbb N$ and all $t\in\mathbb R_{+}$. 
We suppose that $V_i\colon\mathbb R^N\rightarrow \mathbb R^N$, 
$i=1,\ldots,d$, are smooth non-commuting vector fields which 
in coordinates are $V_i=\sum_{j=1}^N V_i^j\partial_{y_j}$.
Repeated iteration of the chain rule reveals the 
stochastic Taylor expansion for the solution to the 
stochastic differential system. Indeed for any smooth function
$f\colon\mR^N\to\mR$ we have the formal stochastic Taylor 
series expansion
\begin{equation*}
f\circ y_t=f\circ y_0+\sum_{w\in\pA}J_w(t)\,V_w\circ f\circ y_0.
\end{equation*}
Here $\pA$ is the collection of non-empty words over the
alphabet $\mathbb A=\{1,\ldots,d\}$. We adopt the standard notation for 
Stratonovich integrals, if $w=a_1\ldots a_n$ then 
\begin{equation*}
J_w(t)=\int_0^t\cdots\int_0^{\tau_{n-1}}
\mathrm{d}W^{a_1}_{\tau_n}\,\cdots\,\mathrm{d}W^{a_n}_{\tau_1}.
\end{equation*}
We have also written the composition of the vector fields as
$V_w\equiv V_{a_1}\circ V_{a_2}\circ\cdots\circ V_{a_n}$.
We define the flow-map $\varphi_t$ as the map such that
\begin{equation*}
\varphi_t\circ f\circ y_0=f\circ y_t.
\end{equation*}
It has a formal stochastic Taylor expansion of the form
\begin{equation*}
\varphi_t=\id+\sum_{w\in\pA}J_w(t)\,V_w.
\end{equation*}
Note that $\varphi_t\circ\varphi_s=\varphi_{t+s}$ for all non-negative $t,s$
and $\varphi_0=\id$, the identity mapping.

Classical strong numerical methods are based on truncating the 
stochastic Taylor expansion for the flow-map and 
applying the resulting approximate flow-map $\hat\varphi_t$ over successive small
subintervals of the global interval of integration required
(see Kloeden and Platen 1999 or Milstein 1994).
An important and expensive ingredient in all numerical methods is the
strong simulation/approximation of the required retained multiple 
integrals $J_w(t)$, on each integration step. Here we will 
take their suitable approximation as granted (see, for example, 
Wiktorsson 2001 for their practical simulation). 
Now let $F\colon\mathrm{Diff}(\mR^N)\to\mathrm{Diff}(\mR^N)$
be a smooth function. We can also construct flow approximations 
from $\varphi_t$ via the following procedure.\smallskip

$\bullet$ Construct the new series $\psi_t=F(\varphi_t)$.

$\bullet$ Truncate this series to produce the finite expansion $\hat\psi_t$.
 
$\bullet$ Reconstruct an approximate flow-map as $\hat\varphi_t=F^{-1}(\hat\psi_t)$.

$\bullet$ The ``flow error'' is the flow remainder $R_t=\varphi_t-\hat\varphi_t$.

$\bullet$ An approximate solution is given by $\hat y_t=\hat\varphi_t\circ y_0$.

$\bullet$ The mean-square error in this approximation is $\|R_t\circ y_0\|_{L^2}^2$.\smallskip

\noindent For the special case $F=\log$, i.e.\ the logarithm function, 
this procedure was outlined by Castell and Gaines (1996).
The resulting series $\psi_t=\log\varphi_t$ 
is the exponential Lie series, which lies in $\mR\la V_1,\ldots,V_d\ra$, 
the non-commutative algebra of formal series generated by the 
vector fields $V_1,\ldots,V_d$.
Indeed any truncation $\hat\psi_t$, with multiple integrals replaced
by suitable approximations, also lies in $\mR\la V_1,\ldots,V_d\ra$
and is therefore a vector field. Hence $\hat\varphi_t=\exp\hat\psi_t$
and an approximation $\hat y_t$ to the solution can be constructed by solving the 
ordinary differential system for $u=u(\tau)$: 
\begin{equation*}
u'=\hat\psi_t\circ u
\end{equation*}
for $\tau\in[0,1]$ with $u(0)=y_0$. The solution 
to this system at time $\tau=1$, itself approximated by an ordinary
differential numerical method, is $u(1)\approx\hat y_t$. 

So far, what has been proved for the Castell--Gaines method?
Castell and Gaines (1995, 1996) prove that
the strong order one-half method constructed in this way is
always more accurate than the Euler--Maruyama method. Indeed
they prove that this method is \emph{asymptotically efficient}
in the sense of Newton (1991). Further in the case of a single
driving Wiener process ($d=1$), they prove the same is true for
the strong order one Castell--Gaines method. By asymptotically
efficient we mean, quoting from Newton (1991), that they 
``minimize the leading coefficient in the expansion of mean-square errors
as power series in the sample step size''.
Lord, Malham and Wiese (2008) and Malham and Wiese (2008)
proved that when the diffusion vector fields \emph{commute}, but not necessarily
with the drift vector field, then the strong order one and also three-halves
Castell--Gaines methods have a mean-square error that is smaller
than the mean-square error for the corresponding stochastic Taylor
method. However Lord, Malham and Wiese (2008) also prove that for
linear diffusion vector fields which do not commute, the
strong order one Castell--Gaines method does not necessarily have a smaller
mean-square error than the corresponding stochastic Taylor method.
Indeed there are regions in the phase space where the local
error of the stochastic Taylor method is smaller.

Hence we are left with the following natural question when the diffusion
vector fields do not commute. Is there a solution series ansatz
$\psi_t=F(\varphi_t)$ for some function $F$, for which
the mean-square error of the numerical method so constructed,
is always smaller than the corresponding stochastic Taylor method?
In this paper we answer this question, indeed under the assumption
there is \emph{no drift}, we:\smallskip

(1) Prove the mean-square error of an approximate
solution constructed from the \emph{sinh-log} expansion 
(with $F=\sinh\log$ above) is smaller than that for the 
stochastic Taylor expansion, \emph{to all orders}. 

(2) Prove that a numerical method based on the sinh-log expansion has 
a global error that is smaller than that for the corresponding stochastic
Taylor method.

(3) Utilize the Hopf shuffle algebra of words underlying such expansions;
in fact we retract to a \emph{new} associative algebra of concatenation 
and shuffle operators, that acts on the Hopf shuffle algebra of words.

(4) Underpin our theoretical results with concrete numerical simulations.\smallskip

\noindent We examine and interpret these statements in detail next, in
Section~\ref{sec:ppleideas}, where we answer the following immediate questions. 
First, what precisely, is the sinh-log approximation and the 
properties we prove for it? Second, how do we prove the result; what is the 
connection with Hopf shuffle algebras and the concatenation-shuffle 
operator algebra mentioned? 
In Section~\ref{sec:cspolys} we provide the technical specification of
the concatenation-shuffle operator algebra and prove
some polynomial identities important for our main result.
In Section~\ref{sec:mainresult} we present our main result.
Then in Section~\ref{sec:prac} we discuss the global error result above, and perform 
numerical simulations confirming our results.
We give some concluding remarks in Section~\ref{sec:conclu}.

%
%
%

\section{Principal ideas}\label{sec:ppleideas}
The goal of this section is to motivate and make precise statements 
about the sinh-log approximation we propose.

\subsection{Stochastic series expansion approximations}
We begin by outlining the approximation procedure presented 
in the introduction in more detail. 
Suppose the smooth function $F$ has a 
real series expansion of the form
\begin{equation*}
F(x)=\sum_{k=1}^\infty C_k\,(x-1)^k,
\end{equation*}
with some finite radius of convergence about $x=1$, and
for some coefficient set $\{C_k\colon k\geqslant1\}$ with $C_1=1$.
Given the flow-map $\varphi_t$, we construct the series
\begin{equation*}
\psi_t=F(\varphi_t)\equiv\sum_{k=1}^\infty C_k\,(\varphi_t-\id)^{k}.
\end{equation*}
We substitute, into this series expansion, the stochastic Taylor series 
for the flow-map $\varphi_t$. After rearrangement, we get
\begin{equation*}
\psi_t=\sum_{w\in\pA} K_w(t)\,V_w,
\end{equation*}
where
\begin{equation*}
K_w(t)=\sum_{k=1}^{|w|} C_k\,\sum_{\substack{u_1,\ldots,u_k\in\pA\\u_1u_2\cdots u_k=w}} 
J_{u_1}J_{u_2}\cdots J_{u_k}(t).
\end{equation*}
We truncate the series, dropping all terms $V_w$ with words 
$w$ of length $|w|\geqslant n+1$. This generates the approximation $\hat\psi_t$,
once we have replaced all retained multiple integrals $J_u(t)$ by suitable
approximations. Then, in principle, we construct the solution approximation 
$\hat y_t$ from $\hat y_t=\hat\varphi_t\circ y_0$, where
\begin{equation*}
\hat\varphi_t=F^{-1}(\hat\psi_t).
\end{equation*}
Performing this reconstruction is nontrivial in general (see Section~\ref{sec:prac}).

For example, to construct the exponential Lie series approximation
of Castell and Gaines (1995), we take $F=\log$ and construct the series
\begin{equation*}
\psi_t=\log\varphi_t\equiv\sum_{k=1}^\infty C_k\,(\varphi_t-\id)^{k},
\end{equation*}
where $C_k=\tfrac{1}{k}(-1)^{k-1}$ for $k\geqslant1$.
Substituting the stochastic Taylor series for $\varphi_t$,
the series expansion for $\psi_t$ above becomes the 
exponential Lie series (see Strichartz 1987, Ben Arous 1989, Castell 1993 
or Baudoin 2004 for the full series)
\begin{equation*}
\psi_t=\sum_{w\in\pA} K_{[w]}(t)\,V_{[w]}
\end{equation*}
where for $w=a_1\ldots a_n$ we have 
$V_{[w]}=[V_{a_1},[V_{a_2},\ldots,[V_{a_{n-1}},V_{a_n}]\ldots]$
and
\begin{equation*}
K_{[w]}=\sum_{\sigma\in\mathbb G_{|w|}}\frac{(-1)^{e(\sigma)}}{|w|^2D_{e(\sigma)}^{|w|-1}}
J_{\sigma^{-1}\circ w}.
\end{equation*}
Here $\mathbb G_{|w|}$ is the group of permutations of the index set
$\{1,\ldots,|w|\}$, $e(\sigma)$ is the cardinality of the set 
$\bigl\{j\in\{1,\ldots,|w|-1\}\colon \sigma(j)>\sigma(j+1)\bigr\}$,
and $D_{e(\sigma)}^{|w|-1}$ is the combinatorial number: $|w|-1$ choose $e(\sigma)$.
Truncating this series and using suitable approximations for the retained 
$J_{\sigma^{-1}\circ w}$, produces $\hat\psi_t$.
We then reconstruct the solution approximately using $\hat\varphi_t=\exp\hat\psi_t$.
The actual solution approximation $\hat y_t=\hat\varphi_t\circ y_0$ 
is then computed by solving the ordinary differential equation generated
by the vector field $\hat\psi_t$.

To construct the sinh-log approximation, we take 
$F=\sinh\log$ so that 
\begin{equation*}
\psi_t=\sinh\log\varphi_t\equiv\tfrac12(\varphi_t-\varphi_t^{-1})
\equiv\sum_{k=1}^\infty C_k\,(\varphi_t-\id)^{k},
\end{equation*}
where $C_1=1$ and $C_k=\tfrac12(-1)^{k-1}$ for $k\geqslant2$. 
Again substituting the stochastic Taylor series for $\varphi_t$ we get
the series expansion for $\psi_t$ shown above with terms $K_w(t)V_w$
where the coefficients $K_w(t)$ now explicitly involve the sinh-log 
coefficients $C_k$. Then, in principle, we can reconstruct 
the solution approximately using
\begin{equation*}
\hat\varphi_t=\exp\sinh^{-1}(\hat\psi_t)\equiv\hat\psi_t+\sqrt{\id+\hat\psi_t^2}.
\end{equation*}

\begin{remark}
Suppose the vector fields $V_i$, $i=1,\ldots,d$ are sufficiently smooth and
$t$ sufficiently small (but finite). Then the approximation $\hat\varphi_t\circ y_0$
constructed using the sinh-log expansion, as just described, is square-integrable. 
Further if $y$ is the exact solution of the stochastic differential equation,
and $\hat\varphi_t$ includes all terms $K_wV_w$ involving words of length $w\leqslant n$,
then there exists a constant $C(n,|y|)$ such that 
$\|y_t-\hat\varphi_t\circ y_0\|_{L^2}\leqslant C(n,|y|)\,t^{(n+1)/2}$; 
here $|\cdot|$ is the Euclidean norm.
This follows by arguments exactly analogous to those for the exponential Lie
series given in Malham \& Wiese (2008; Theorem~7.1 and Appendix~A).
\end{remark}

\begin{remark}
Naturally, the exponential Lie series $\psi_t$ and and its truncation $\hat\psi_t$
lie in the Lie algebra of vector fields generated by $V_1,\ldots,V_d$.
Hence $\exp(\tau\hat\psi_t)$ is simply the ordinary flow-map associated with
the autonomous vector field $\hat\psi_t$. 
\end{remark}

\begin{remark}
The exponential Lie series originates with Magnus (1954) and Chen (1957); see 
Iserles (2002). In stochastic approximation it appears in 
Kunita (1980), Fleiss (1981), Azencott (1982),
Strichartz (1987), Ben Arous (1989), Castell (1993), 
Castell \& Gaines (1995,1996), Lyons (1998), Burrage \& Burrage (1999), 
Baudoin (2004), Lord, Malham \& Wiese (2009) and Malham \& Wiese (2008).
\end{remark}

\subsection{Hopf algebra of words}
Examining the coefficients $K_w$ in the series expansion for $\psi$
above, we see that they involve linear combinations
of products of multiple Stratonovich integrals (we suspend explicit
$t$ dependence momentarily). The question is, can we determine $K_w$ explicitly? 
Our goal here is to reduce this problem to a pure combinatorial one. 
This involves the Hopf algebra of words (see Reutenauer 1993).

Let $\mK$ be a commutative ring with unit. 
In our applications we take $\mK=\mR$ or $\mK=\mJ$,
the ring generated by multiple Stratonovich integrals and the constant
random variable $1$, with pointwise multiplication and addition.
Let $\mK\la\mA\ra$ denote the set of all noncommutative polynomials
and formal series on the alphabet $\mA=\{1,2,\ldots,d\}$ over $\mK$.
With the concatenation product, $\mK\la\mA\ra$ is the 
associative \emph{concatenation algebra}.
For any two words $u,v\in\mK\la\mA\ra$ with lengths $|u|$ and $|v|$, 
we define the \emph{shuffle product} $u\sh v$ to be the sum of the 
words of length $|u|+|v|$ created by shuffling all the letters 
in $u$ and $v$ whilst preserving their original order.
The shuffle product is extended to $\mK\la\mA\ra$ by bilinearity.
It is associative, and distributive with
respect to addition, and we obtain on $\mK\la\mA\ra$
a commutative algebra called the \emph{shuffle algebra}
(note $\ml\sh w=w\sh \ml=w$ for any word $w$ where $\ml$ is the empty word).
The linear signed reversal mapping $\alpha\in\mathrm{End}(\mK\la\mA\ra)$:
\begin{equation*}
\alpha\circ w=(-1)^na_n\ldots a_1,
\end{equation*}
for any word $w=a_1\ldots a_n$ is the \emph{antipode} on $\mK\la\mA\ra$.
There are two \emph{Hopf algebra} structures
on $\mK\la\mA\ra$, namely $(\mK\la\mA\ra,c,\delta,\eta,\varepsilon,\alpha)$ 
and $(\mK\la\mA\ra,s,\delta',\eta,\varepsilon,\alpha)$, where
$\eta$ and $\varepsilon$ are unit and co-unit elements, and 
$\delta$ and $\delta'$ respective co-products (Reutenauer, p.~27).
We define the associative algebra using the complete tensor product 
\begin{equation*}
\mathcal H=\mK\la\mA\ra\overline\otimes\,\mK\la\mA\ra
\end{equation*}
with the shuffle product on the left and the concatenation
product on the right (Reuntenauer, p.~29). 
The product of elements 
$u\otimes x,v\otimes y\in\mathcal H$ is given by
\begin{equation*}
(u\otimes x)(v\otimes y)=(u\sh v)\otimes(xy),
\end{equation*}
and formally extended to infinite linear combinations in $\mathcal H$
via linearity. As a tensor product of two Hopf algebra structures,
$\mathcal H$ itself acquires a Hopf algebra structure.

\subsection{Pullback to Hopf shuffle algebra}
Our goal is to pullback the flow-map 
$\varphi$ and also $\psi$ to $\mathcal H$ (with $\mK=\mR$). 
Let $\mV$ be the set of all vector fields
on $\mR^N$; it is an $\mR$-module over $C^\infty\bigl(\mR^N\bigr)$
(see Varadarajan 1984, p.~6). We know that for the stochastic
Taylor series the flow-map $\varphi\in\mJ\la\mV\ra$ (with
vector field composition as product).
Since $\mJ\la\mV\ra\cong\bigoplus_{n\geqslant0}\mJ\otimes\mV_n$,
where $\mV_n$ is the subset of $\mV$ of compositions of
vector fields of length $n$, we can write
\begin{equation*}
\varphi=1\otimes\id_{\mV}+\sum_{w\in\pA}J_w\otimes V_w.
\end{equation*}
The linear \emph{word-to-vector field map} 
$\kappa\colon\mR\la\mA\ra\rightarrow\mV$
given by $\kappa\colon\omega\mapsto V_\omega$
is a concatenation homomorphism, i.e.\ $\kappa(uv)=\kappa(u)\kappa(v)$
for any $u,v\in\mA^+$. 
And the linear \emph{word-to-integral map}
$\mu\colon\mR\la\mA\ra\rightarrow\mJ$ given by $\mu\colon\omega\mapsto J_\omega$
is a shuffle homomorphism, i.e.\ $\mu(u\sh v)=\mu(u)\mu(v)$ for any $u,v\in\mA^+$
(see for example, Lyons, \emph{et. al.\ } 2007, p.~35 
or Reutenauer 1993, p.~56). Hence the map 
$\mu\otimes\kappa\colon\mathcal H\rightarrow\bigoplus_{n\geqslant0}\mJ\otimes\mV_n$
is a Hopf algebra homomorphism.
The pullback of the flow-map $\varphi$ by $\mu\otimes\kappa$ is
\begin{equation*}
(\mu\otimes\kappa)^*\varphi=\ml\otimes\ml+\sum_{w\in\pA}w\otimes w.
\end{equation*}
All the relevant information about the stochastic flow 
is encoded in this formal series; it is essentially Lyons' \emph{signature} 
(see Lyons, Caruana and L\'evy 2007; Baudoin 2004).
Hence by direct computation, formally we have
\begin{align*}
(\mu\otimes\kappa)^*\psi
=&\;\sum_{k\geqslant 1}C_k\bigl((\mu\otimes\kappa)^*\varphi-\ml\otimes\ml\bigr)^k\\
=&\;\sum_{k\geqslant 1}C_k\Biggl(\sum_{w\in\pA}w\otimes w\Biggr)^k\\
=&\;\sum_{k\geqslant 1}C_k\Biggl(\sum_{u_1,\ldots,u_k\in\pA}(u_1\sh\ldots\sh u_k)\otimes
(u_1\ldots u_k)\Biggr)\\
=&\;\sum_{w\in\mA^*}\Biggl(\sum_{k=1}^{|w|}
C_k\sum_{\substack{u_1,\ldots,u_k\in\pA\\w=u_1\ldots u_k}}u_1\sh\ldots\sh u_k\Biggr)\otimes w\\
=&\;\sum_{w\in\mA^*}(K\circ w)\otimes w,
\end{align*}
where $K\circ w$ is defined by
\begin{equation*}
K\circ w=\sum_{k=1}^{|w|}
C_k\sum_{\substack{u_1,\ldots,u_k\in\pA\\w=u_1\ldots u_k}}u_1\sh\ldots\sh u_k,
\end{equation*}
corresponds to $K_w$ (indeed it is the pullback $\mu^*K_w$ to $\mO_w$; 
see Section~\ref{sec:cspolys}).
Having reduced the problem of determining $K\circ w$ to the algebra of shuffles,
a further simplifying reduction is now possible.

\begin{remark}
The use of Hopf shuffle algebras
in stochastic expansions can be traced through, for example,
Strichartz (1987), Reutenauer (1993), Gaines (1994), Li \& Liu (2000), 
Kawski (2001), Baudoin (2004), Murua (2005), 
Ebrahimi--Fard \& Guo (2006), Manchon \& Paycha (2006) and 
Lyons, Caruana \& L\'evy (2007), to name a few. 
The paper by Munthe--Kaas \& Wright (2008) on the Hopf algebraic 
of Lie group integrators actually instigated the Hopf algebra direction 
adopted here. A useful outline on the use of 
Hopf algebras in numerical analysis can be found therein, as
well as the connection to the work by 
Connes \& Miscovici (1998) and Connes \& Kreimer (1998) in 
renormalization in perturbative quantum field theory.
\end{remark}

\subsection{Retraction to concatenations and shuffles}
For any given word $w=a_1\ldots a_{n+1}$ we now focus on 
the coefficients $K\circ w$. We observe that it is the concatenation and 
shuffle operations encoded in the structural form of the sum for $K\circ w$
that carry all the relevant information. Indeed each term 
$u_1\sh\ldots\sh u_k$ is a partition of $w$ into subwords that 
are shuffled together. Each subword $u_i$ is a concatenation of 
$|u_i|$ letters, and so we can reconstruct each term of the
form $u_1\sh\ldots\sh u_k$ from the following sequence applied to the 
word $w$:
\begin{equation*}
c^{|u_1|-1}sc^{|u_2|-1}s\ldots sc^{|u_k|-1}
\end{equation*}
where the power of the letter $c$ indicates the number of letters concatenated
together in each subword $u_i$ and the letter $s$ denotes the shuffle product
between the subwords. In other words, if we factor out the word $w$, we
can replace $K\circ w$ by a polynomial $K$ of the letters $c$ and $s$. 
In fact, in Lemma~\ref{lem:convert} we show that
\begin{equation*}
K=\sum_{k=0}^nC_{k+1}(c^{n-k}\sh s^k).
\end{equation*}
Thus we are left with the task of simplifying this polynomial in two
variables (lying in the real associative algebra of concatenation 
and shuffle operations). 

\begin{remark}
We devote Section~\ref{sec:cspolys} to the rigorous justification of
this retraction, the result above, and those just following.
A key ingredient is to identify the correct action of this algebra over 
$\bigl(\mK\la\mA\ra,s,\alpha\bigr)$.
\end{remark}

\begin{remark}
There is a natural right action by the symmetric group $\mS_n$ 
on $\mK\la\mA\ra_n$, the subspace of $\mK\la\mA\ra$ spanned by words of length $n$
(Reutenauer 1993, Chapter~8).
This action is transitive and extends by linearity to
a right action of the group algebra $\mK\la\mS_n\ra$ on $\mK\la\mA\ra_n$.
We are primarily concerned with shuffles and multi-shuffles,
a subclass of operations in $\mK\la\mS_n\ra$, and in particular,
we want a convenient structure that enables us to combine
single shuffles to produce multi-shuffles.
\end{remark}

\subsection{Stochastic sinh-log series coefficients}
The coefficient set $\{C_k\colon k\geqslant1\}$ determines the 
form of the function $F$. 
Our ultimate goal is to show order by order that the stochastic
sinh-log expansion guarantees superior accuracy. Hence order
by order we allow a more general coefficient set $\{C_k\colon k\geqslant1\}$, 
and show that the sinh-log choice provides the guarantee we seek.

\begin{definition}[Partial sinh-log coefficient set]
Define the partial sinh-log coefficient sequence:
\begin{equation*}
C_k=\begin{cases}
     1, & \quad k=1,\\
    \tfrac12(-1)^{k-1}, & \quad k\geqslant2,\\
    \tfrac12(-1)^{n}+\epsilon, & \quad k=n+1,
    \end{cases}
\end{equation*}
where $\epsilon\in\mathbb R$.
\end{definition}

With the choice of coefficients $\{C_k\}$, we see that
\begin{equation*}
K=\tfrac12 c^n+\tfrac12\sum_{k=0}^n(-1)^k(c^{n-k}\sh s^k)+\epsilon s^n.
\end{equation*}
This has an even simpler form. 

\begin{lemma}\label{lem:sinhlogcoeffs}
With the partial sinh-log coefficient sequence $\{C_k\}$, the 
coefficient $K$ is given by
\begin{equation*}
K=\tfrac12\bigl(c^n-\alpha_n\bigr)+\epsilon\,s^n,
\end{equation*}
where $\alpha_n$ is the antipode for words of length $n+1$.
\end{lemma}

\begin{proof}
We think of $(c-s)^n$, with expansion by concatenation, 
as the generator for the polynomial 
(in $\mK\la\mB\ra_n$; see Section~\ref{sec:cspolys}) 
defined by  
\begin{equation*}
(c-s)^n=\sum_{k=0}^n (-1)^k (c^{n-k}\sh s^k).
\end{equation*}
Then by Lemma~\ref{lem:antipodepoly} in Section~\ref{sec:cspolys} we have the 
following identity
\begin{equation*}
(c-s)^n\equiv-\alpha_n.
\end{equation*}
Hence using the sinh-log coefficients and splitting the
first term, we have  
\begin{align*}
K=&\;\tfrac12 c^n+\tfrac12\sum_{k=0}^{n}(-1)^k(c^{n-k}\sh s^{k})+\epsilon\,s^n\\
=&\;\tfrac12 c^n+\tfrac12(c-s)^n+\epsilon\,s^n\\
=&\;\tfrac12 c^n-\tfrac12\alpha_n+\epsilon\,s^n.
\end{align*}
\end{proof}

\begin{corollary}\label{cor:coeffform}
For any word $w=a_1\ldots a_{n+1}$ we have
\begin{equation*}
K\circ w=\tfrac12\bigl(w-\alpha\circ w\bigr)+\epsilon\,a_1\sh\ldots\sh a_{n+1},
\end{equation*}
and thus
\begin{equation*}
K_w=\tfrac12\bigl(J_w-J_{\alpha\circ w}\bigr)+\epsilon\prod_{i=1}^{n+1}J_{a_i}.
\end{equation*}
\end{corollary}

\begin{remark}
For the stochastic sinh-log expansion the coefficients $K_w$
thus have an extremely simple form. 
There are several strategies to prove this form.
The result can be proved directly in terms of multiple 
Stratonovich integrals by judicious use of their properties,
the partial integration formula and induction---the proof is 
long but straightforward. That this strategy 
works is also revealed by the strategy we have adopted in 
this paper, which we believe is shorter and more insightful.
\end{remark}

\section{Concatenation-shuffle operator algebra}\label{sec:cspolys}
\subsection{Algebra and action}
With $\mB=\{c,s\}$, let $\mK\la\mB\ra$ denote the set of all noncommutative 
polynomials and formal series on $\mB$ over $\mK$. 
We can endow $\mK\la\mB\ra$ with the concatenation 
product or shuffle product, and also generate an associative 
\emph{concatenation-shuffle operator algebra} on 
$\mK\la\mB\ra$ as follows.

\begin{definition}[Shuffle gluing product]
The map $g\colon\mK\la\mB\ra\otimes\mK\la\mB\ra\rightarrow\mK\la\mB\ra$,
the associative and bilinear shuffle gluing product, is defined by
\begin{equation*}
g\colon b_1\otimes b_2\mapsto b_1sb_2,
\end{equation*}
i.e.\ we concatenate the element $b_1$ with $s$ and the element $b_2$ 
in $\mK\la\mB\ra$ as shown. 
\end{definition}

Endowed with the shuffle gluing product, $\mK\la\mB\ra$ is
an associative algebra with unit element $s^{-1}$ 
(see Reutenauer 1993, p.~26, for the definition of $s^{-1}$).
We define the graded associative tensor algebra $\cK$ by
\begin{equation*}
\cK=\bigoplus_{n\geqslant0}\mK\la\mB\ra_n\otimes\mK\la\mA\ra_{n+1},
\end{equation*}
with the shuffle gluing product on the left in $\mK\la\mB\ra_n$ 
and concatenation product on the right in $\mK\la\mA\ra_{n+1}$---here
$\mK\la\mB\ra_n$ and $\mK\la\mA\ra_{n+1}$ denote the subspaces of 
$\mK\la\mB\ra$ and $\mK\la\mA\ra$, respectively, spanned by words
of length $n$ and $n+1$, respectively.
Thus if $b_1\otimes u_1$ and $b_2\otimes u_2$ 
are in $\cK$ then their product is
\begin{equation*}
(b_1\otimes u_1)(b_2\otimes u_2)=(b_1sb_2)\otimes(u_1u_2),
\end{equation*}
with extension to $\cK$ by bilinearity.

We now define the homomorphism 
$\zeta\colon\cK\rightarrow(\mK\la\mA\ra,s,\delta',\alpha)$ 
as follows. Any word $b\in\mB^+$,
for some $k\in\mathbb N$ and $n_1,\ldots,n_k\in\mathbb N\cup\{0\}$,
can be expressed in the form
\begin{equation*}
b=c^{n_1}sc^{n_2}sc^{n_3}\ldots sc^{n_k}.
\end{equation*}
There are $(k-1)$ occurrences of the symbol `$s$' in $b$, 
and $n_1+n_2+\cdots+n_k+k-1=|b|$. Here $c^n$ represents
the word consisting of $c$ multiplied by concatenation $n$ times,
$c^0=1$; similarly for $s^n$ and $s^0$. Then we define
\begin{equation*}
\zeta\colon b\otimes w\mapsto b\circ w=u_{n_1}\sh u_{n_2}\sh\ldots\sh u_{n_k},
\end{equation*}
where $w=u_{n_1}u_{n_2}\ldots u_{n_k}$ and the successive
subwords $u_{n_1}$, $u_{n_2}$,\ldots,$u_{n_k}$ have 
respective lengths $n_1+1$, $n_2+1$,\ldots, $n_k+1$.  
Note the sum of the lengths of the subwords is $n+1$.
The map $\zeta$ extends by linearity to $\cK$.
The $c$-symbol indicates a concatenation
product and the $s$-symbol a shuffle product in the appropriate
$n$ slots between the $n+1$ letters in any word $w$ on $\mA^+$ 
of length $n+1$. That $\zeta$ is a homomorphism from 
$\cK$ to $\mK\la\mA\ra$ follows from:
\begin{equation*}
\zeta\bigl((b_1\otimes u_1)(b_2\otimes u_2)\bigr)
=\zeta\bigl((b_1sb_2)\otimes(u_1u_2)\bigr)
=\zeta(b_1\otimes u_1)\sh\zeta(b_2\otimes u_2).
\end{equation*}

\begin{definition}[Partition orbit]
We define the (shuffle) \emph{partition orbit}, $\mO(w)$, 
of a word $w\in\pA$ to be the subset of $\mK\la\mA\ra$
whose elements are linear combinations of words
constructed by concatenating and shuffling the 
letters of $w=a_1\ldots a_n$:
\begin{equation*}
\mO(w)=\bigl\{\mathrm{span}(u_1\sh\ldots\sh u_k)\colon 
u_1\ldots u_k=w;~u_1,\ldots,u_k\in\pA;~k\in\{1,\ldots,|w|\}\bigr\}.
\end{equation*}
\end{definition}

For any $u\in\mO(w)$ there exists a $b\in\mK\la\mB\ra_{|w|-1}$
such that $u=\zeta(b\otimes w)$. Hence we can consider the 
preimage of $\mO(w)$ under $\zeta$ in $\cK$ 
given by $\zeta^{-1}\mO(w)=\{b\otimes w\in\cK\colon \zeta(b\otimes w)\in\mO(w)\}$.
Thus any element in $\mO(w)$ can be identified with an
element $b\otimes w\in\zeta^{-1}\mO(w)$ for a unique $b\in\mK\la\mB\ra_{|w|-1}$
and there is a natural projection map
\begin{equation*}
\pi\colon\mO(w)\rightarrow\mK\la\mB\ra_{|w|-1}.
\end{equation*}

\subsection{Polynomial identities}
\label{sec:polyids}
Here we prove a sequence of lemmas that combine to 
prove our main results. 
The aim of the first two lemmas is to establish a
form for the antipode $\alpha$ as a polynomial in the concatenation-shuffle 
operator algebra $(\mR\la\mB\ra,g)$.
We shall denote the antipode in $\mathrm{End}\bigl(\mR\la\mA\ra_{n+1}\bigr)$
by $\alpha_n$; it sign reverses any word $w\in\mR\la\mA\ra_{n+1}$.

\begin{lemma}[Partial integration formula] \label{lem:partint}
The partial integration formula applied repeatedly to 
the multiple Stratonovich integral $J_w$, where $w=a_1\ldots a_{n+1}$,
pulled back to $\mR\la\mB\ra_{|w|-1}$, is given by
\begin{equation*}
\alpha_n\equiv-c^n-\sum_{k=0}^{n-1}c^ks\alpha_{n-k-1}.
\end{equation*}
\end{lemma}

\begin{proof}
Repeated partial integration on the multiple 
Stratonovich integral $J_w$ with $w=a_1\ldots a_{n+1}$,
pulled back to $\mR\la\mA\ra_{n+1}$ via the word-to-integral map 
$\mu$ generates the identity:
\begin{equation*}
a_1\ldots a_{n+1}=(a_1\ldots a_n)\sh a_{n+1}-(a_1\ldots a_{n-1})\sh (a_{n+1}a_n)
+\cdots+(-1)^na_{n+1}\ldots a_1.
\end{equation*}
After rearrangement, the projection of this identity in $\mO(w)$ 
onto $\mR\la\mB\ra_n$ via $\pi$, using the definition for
$\alpha_n$, generates the identity shown.
\end{proof}
 
\begin{lemma}[Antipode polynomial]\label{lem:antipodepoly}
The antipode $\alpha_n\in\mathrm{End}(\mR\la\mA\ra_{n+1})$
and polynomial $-(c-s)^n\in\mK\la\mB\ra_n$
are the same linear endomorphism on $\mR\la\mA\ra_{n+1}$:
\begin{equation*}
\alpha_n\equiv-(c-s)^n.
\end{equation*}
\end{lemma}

\begin{proof}
The statement of the lemma is trivially true for $n=1,2$.
We assume it is true for $k=1,2,\ldots,n-1$.
Direct expansion reveals that
\begin{equation*}
(c-s)^n=c^n-\sum_{k=0}^{n-1}c^ks(c-s)^{n-k-1}.
\end{equation*}
By induction, using our assumption in this expansion
and comparing with the partial integration formula 
in Lemma~\ref{lem:partint} proves the statement for $k=n$.
\end{proof}

\begin{lemma}
\label{lem:convert}
The projection of $K\circ w\in\mO(w)$ onto $\mR\la\mB\ra_n$ via $\pi$
generates 
\begin{equation*}
K=\sum_{k=0}^{n}C_{k+1} (c^{n-k}\sh s^{k}).
\end{equation*}
\end{lemma}

\begin{proof}
For any $w\in\pA$ with $|w|=n+1$ we have
\begin{align*}
K\circ w=&\;\sum_{k=1}^{n+1}C_k\,\zeta\bigl((c^{n-(k-1)}\sh s^{k-1})\otimes w\bigr)\\
=&\;\sum_{k=0}^{n}C_{k+1}\,\zeta\bigl((c^{n-k}\sh s^{k})\otimes w\bigr)\\
=&\;\zeta\Biggl(\Biggl(\sum_{k=0}^{n}C_{k+1}(c^{n-k}\sh s^{k})\Biggr)\otimes w\Biggr).
\end{align*}
Projecting this onto $\mR\la\mB\ra_n$ establishes the result.
\end{proof}

\section{Mean-square sinh-log remainder is smaller}\label{sec:mainresult}
Suppose the \emph{flow remainder} associated with a flow approximation 
$\hat\varphi_t$ is
\begin{equation*}
R_t:=\varphi_t-\hat\varphi_t.
\end{equation*}
The remainder associated with the approximation 
$\hat y_t=\hat\varphi_t\circ y_0$ 
is thus $R_t\circ y_0$. We measure the error in this
approximation, for each $y_0\in\mR^N$, in mean-square by
\begin{equation*}
\|R_t\circ y_0\|_{L^2}^2
=\mE\bigl((R_t\circ y_0)^{\text{\tiny T}}(R_t\circ y_0)\bigr).
\end{equation*}
If we truncate $\psi=F(\varphi)$ to $\hat\psi$, including 
all terms $V_w$ with words of length $|w|\leqslant n$, 
suppose the remainder is $r$, i.e.\ we have
\begin{equation*}
\psi=\hat\psi+r.
\end{equation*}
Then the remainder to the corresponding 
approximate flow $\hat\varphi^{\text{sl}}=F^{-1}(\hat\psi)$,
taking the difference with the exact 
stochastic Taylor flow $\varphi^{\text{st}}$, is given by
\begin{align*}
R^{\text{sl}}=&\;\varphi^{\text{st}}-\hat\varphi^{\text{sl}}\\
=&\;F^{-1}(\psi)-F^{-1}(\hat\psi)\\
=&\;F^{-1}(\hat\psi+r)-F^{-1}(\hat\psi)\\
=&\;r-C_2(r\hat\psi+\hat\psi r)+\mathcal O(\hat\psi^2r),
\end{align*}
where we can ignore the $C_2\bigl(r\hat\psi+\hat\psi r\bigr)$ 
and $\mathcal O(\hat\psi^2r)$ terms in this section---we comment 
on their significance in Section~\ref{sec:prac}. 
We compare this with the stochastic Taylor flow remainder 
$R^{\text{st}}:=\varphi^{\text{st}}-\hat\varphi^{\text{st}}$,
where $\hat\varphi^{\text{st}}$ is the stochastic Taylor flow
series truncated to include all terms $V_w$ with words of 
length $|w|\leqslant n$. Indeed, we set
\begin{equation*}
\bar R:= R^{\text{st}}-R^{\text{sl}},
\end{equation*}
and use the $L^2$ norm to measure the remainder. Hence
for any data $y_0$, we have
\begin{equation*}
\|R^{\text{st}}\circ y_0\|_{L^2}^2=\|R^{\text{sl}}\circ y_0\|_{L^2}^2+E,
\end{equation*}
where the \emph{mean-square excess}
\begin{equation*}
E:=
\mathbb E\,\bigl(\bar R\circ y_0\bigr)^{\text{\tiny T}}\bigl(R^{\text{sl}}\circ y_0\bigr)
+\mathbb E\,\bigl(R^{\text{sl}}\circ y_0\bigr)^{\text{\tiny T}}\bigl(\bar R\circ y_0\bigr)
+\mathbb E\,\bigl(\bar R\circ y_0\bigr)^{\text{\tiny T}}\bigl(\bar R\circ y_0\bigr).
\end{equation*}
If $E$ is positive then $R^{\text{sl}}\circ y_0$ is smaller than 
$R^{\text{st}}\circ y_0$ in the $L^2$ norm.

\begin{theorem}
Suppose we construct the finite sinh-log expansion $\psi$ using 
the \emph{partial sequence of sinh-log coefficients} $\{C_k\}$,
and truncate $\psi$ producing $\hat\psi$ 
which only includes terms with words $w$ with $|w|\leqslant n$.
Then the flow remainders for the sinh-log and the corresponding stochastic Taylor
approximations are such that, for any data $y_0$ and order $n\in\mathbb N$,
the mean-square excess is given by
\begin{equation*}
E=E_0-\epsilon E_1-\epsilon^2 E_2,
\end{equation*}
where $E_0>0$, $E_2>0$ and 
\begin{equation*}
E_1=\begin{cases}
   \hat E_1, &\text{if}~n~\text{even},\\
   0, &\text{if}~n~\text{odd},
   \end{cases}
\end{equation*}
where, for $n$ even, we have 
\begin{equation*}
\hat E_1=\sum_{\substack{u,v\in\mathbb A^{+}\\|u|=|v|=n+1}}
\xi(u,v)\,(V_u\circ y_0)^{\mathrm{T}}(V_v\circ y_0),
\end{equation*}
and 
\begin{equation*}
\xi(u,v)=\mE\,\Biggl(\tfrac12(J_u+J_{\rho\circ u})\prod_{i=1}^{n+1}J_{v_i}
              +\tfrac12(J_v+J_{\rho\circ v})\prod_{i=1}^{n+1}J_{u_i}\Biggr)>0.
\end{equation*}
Here $\rho$ is the unsigned reversal mapping, i.e.\ if $w=a_1\ldots a_n$ 
then $\rho\circ w=a_n\ldots a_1$.
\end{theorem}

\begin{proof}
If we truncate the sinh-log series flow-map including all integrals
associated with words of length $n$, the remainder is given by
\begin{equation*}
R^{\text{sl}}=\sum_{\substack{w\in\pA\\|w|\geqslant n+1}}K_w\,V_w+\cdots,
\end{equation*}
where henceforth we will ignore integrals in the remainder 
with $|w|\geqslant n+2$. Recall that from Corollary~\ref{cor:coeffform}
we have
\begin{equation*}
K_w=\tfrac12\bigl(J_w-J_{\alpha\circ w}\bigr)+\epsilon\prod_{i=1}^{n+1}J_{w_i}.
\end{equation*}
The corresponding stochastic Taylor flow-map remainder is
\begin{equation*}
\sum_{\substack{w\in\mathbb A^{+}\\|w|=n+1}}J_w\,V_w.
\end{equation*}
The difference between the two is
\begin{equation*}
\bar R=\sum_{\substack{w\in\mathbb A^{+}\\|w|=n+1}}\bar J_w\,V_w,
\end{equation*}
where $\bar J_w=J_w-K_w$ and is given by
\begin{equation*}
\bar J_w=\tfrac12\bigl(J_w+J_{\alpha\circ w}\bigr)
-\epsilon\prod_{i=1}^{n+1}J_{w_i}.
\end{equation*}
The mean-square excess to the sinh-log remainder
is $E$ which at leading order is 
\begin{equation*}
\sum_{\substack{u,v\in\mathbb A^{+}\\|u|=|v|=n+1}}
\mathbb E\,\bigl(\bar J_uK_v+K_u\bar J_v
+\bar J_u\bar J_v\bigr)
\,(V_u\circ y_0)^{\text{\tiny T}}(V_v\circ y_0).
\end{equation*}
We need to determine whether this quantity
is positive definite or not. 
We refer to $\bar J_uK_v+K_u\bar J_v$ as the 
\emph{cross-correlation} terms and $\bar J_u\bar J_v$ 
as the \emph{auto-correlation} terms.
The forms for $K_u$ and $\bar J_u$ imply that:
\begin{align*}
\bar J_uK_v+K_u\bar J_v+\bar J_u\bar J_v
=&\;\epsilon^0\Bigl(\tfrac12(J_u J_v-J_{\alpha\circ u}J_{\alpha\circ v})+
\tfrac14(J_u+J_{\alpha\circ u})(J_v+J_{\alpha\circ v})\Bigr)\\
&\;-\epsilon^1\Biggl(\tfrac12(J_u-J_{\alpha\circ u})\prod_{i=1}^{n+1}J_{v_i}
+\tfrac12(J_v-J_{\alpha\circ v})\prod_{i=1}^{n+1}J_{u_i}\Biggr)\\
&\;-\epsilon^2\Biggl(\prod_{i,j=1}^{n+1}J_{u_i}J_{v_j}\Biggr).
\end{align*}
Consider the zero order $\epsilon^0$ term.
Using Lemma~\ref{lem:expdiff}, the expectation of the 
cross-correlation term therein (the first term on the right above) 
is zero. Hence we have 
\begin{align*}
E_0&=\sum_{\substack{u,v\in\mathbb A^{+}\\|u|=|v|=n+1}}
\mE\bigl(\tfrac14(J_u+J_{\alpha\circ u})(J_v+J_{\alpha\circ v})\bigr)\,
(V_u\circ y_0)^{\text{\tiny T}}(V_v\circ y_0)\\
&=\mE\Bigg(\sum_{|u|=n+1}\tfrac12(J_u+J_{\alpha\circ u})\,(V_u\circ y_0)\Biggr)^{\text{\tiny T}}
\Bigg(\sum_{|v|=n+1}\tfrac12(J_v+J_{\alpha\circ v})\,(V_v\circ y_0)\Biggr)\\
&>0.
\end{align*}
At the next order $\epsilon^1$,
the terms shown are solely from cross-correlations---with
the auto-correlation terms cancelling with other 
cross-correlation terms. When $n$ is odd the 
expectation of this term is zero, again using Lemma~\ref{lem:expdiff}.
When $n$ is even we get the expression for $E_1$ stated in the theorem.
Finally at order $\epsilon^2$ the coefficient shown is 
the auto-correlation term multiplied by minus one. Explicitly we see that
\begin{align*}
E_2&=\sum_{\substack{u,v\in\mathbb A^{+}\\|u|=|v|=n+1}}
\mE\Biggl(\prod_{i,j=1}^{n+1}J_{u_i}J_{v_j}\Biggr)\,(V_u\circ y_0)^{\text{\tiny T}}(V_v\circ y_0)\\
&=\mE\Bigg(\sum_{|u|=n+1}\prod_{i=1}^{n+1}J_{u_i}\,(V_u\circ y_0)\Biggr)^{\text{\tiny T}}
\Bigg(\sum_{|v|=n+1}\prod_{j=1}^{n+1}J_{v_j}\,(V_v\circ y_0)\Biggr)\\
&>0.
\end{align*}
Combining these results generates the form for $E$ stated.
\end{proof}

\begin{corollary}\label{cor:oddeven}
When $n$ is odd, $E$ is positive and maximized when
$\epsilon=0$. When $n$ is even, it is positive at $\epsilon=0$,
but maximized at a different value of~$\epsilon$; the maximizing
value will depend on the vector fields.
\end{corollary}

\begin{lemma}\label{lem:expdiff}
For any pair $u,v\in\pA$ we have that
\begin{equation*}
\mathbb E\bigl(J_u J_v-J_{\rho\circ u}J_{\rho\circ v}\bigr)=0.
\end{equation*}
\end{lemma}

\begin{proof}
Every Stratonovich integral $J_w$ is a linear
combination of It\^o integrals
\begin{equation*}
J_w=\sum_{u\in\mathbb D(w)} c_u I_u\,,
\end{equation*}
where the set $\mathbb D(w)$ consists of $w$ and all multi-indices
$u$ obtained by successively replacing any two adjacent
(non-zero) equal indices in $w$ by $0$,
see Kloeden and Platen (1999), equation (5.2.34). 
Since all indices in $w$ are non-zero by assumption, 
the constant $c_u$ is given by
\begin{equation*}
c_u=\bigl(\tfrac12\bigr)^{\nn(u)},
\end{equation*}
where $\nn(u)$ denotes the number of zeros in $u$. Since
adjacency is retained when reversing an index, it follows that
\begin{equation*}
J_{\rho\circ w}=\sum_{u\in\mathbb D(w)} c_u I_{\rho\circ u}.
\end{equation*}
Lemma 5.7.2 in Kloeden and Platen (1999) implies that 
the expected value of the product of two multiple 
It\^o integrals is of the form
\begin{equation*}
\mathbb E\bigl(I_uI_v\bigl)
=f\biggl(\ell(u),\ell(v),
\sum_{i=0}^{\ell(u)}\bigl(k_i(u)+k_i(v)\bigr),
\prod_{i=0}^{\ell(u)} 
\frac{\bigl(k_i(u)+k_i(v)\bigr)!}{k_i(u)!k_i(v)!}\biggr)
\end{equation*}
for some function $f$. Here $\ell(u)$
denotes the number of non-zero indices in $u$, while $k_0(u)$
denotes the number of zero components preceding the first
non-zero component of $u$, and $k_i(u)$, for $i=1,\ldots,\ell(u)$,
the number of components of $u$ between the $i$-th and $(i+1)$-th 
non-zero components, or the end of $u$ if $i=\ell(u)$. 
It follows that
\begin{equation*}
k_i(u)=k_{\ell(u)-i}(\rho\circ u).
\end{equation*}
Since all other operations in the arguments of $f$ 
are unchanged by permutations in
$u$ and $v$, we deduce that
\begin{equation*}
\mathbb E\bigl(I_uI_v\bigr)
=\mathbb E\bigl(I_{\rho\circ u}I_{\rho\circ v}\bigr),
\end{equation*}
and consequently,
\begin{equation*}
\mathbb E\bigl(J_uJ_v-J_{\rho\circ u}J_{\rho\circ v}\bigr)
=\sum_{\substack{u'\in\mathbb D(u)\\v'\in\mathbb D(v)}} c_{u'}c_{v'}\,
\mathbb E\bigl(I_{u'} I_{v'}-I_{\rho\circ u'}I_{\rho\circ v'}\bigr)=0.
\end{equation*}
\end{proof}

\section{Practical implementation}\label{sec:prac}
\subsection{Global error}
We define the \emph{strong global error} associated with an
approximate solution $\hat y_T$ over the global interval
$[0,T]$ as $\mathcal E:=\|y_T-\hat y_T\|_{L^2}$.
Suppose the exact $y_T$ and approximate solution $\hat y_T$ 
are constructed by successively applying the exact and 
approximate flow-maps $\varphi_{t_m,t_{m+1}}$ and $\hat\varphi_{t_m,t_{m+1}}$
on $M$ successive intervals $[t_m,t_{m+1}]$, with $t_m=mh$ for $m=0,1,\ldots,M-1$ 
and $h=T/M$ as the fixed stepsize, to the initial data $y_0$.
A straightforward calculation shows that up to higher order
terms we have 
\begin{equation*}
\mathcal E^2=\mE\,\bigl(\mathcal R\circ y_0\bigr)^{\text{\tiny T}}\bigl(\mathcal R\circ y_0\bigr),
\end{equation*}
where 
\begin{equation*}
\mathcal R\circ y_0\equiv
\sum_{m=0}^{M-1}\hat\varphi_{t_{m+1},t_M}\circ R_{t_m,t_{m+1}}\circ \hat\varphi_{t_0,t_m},
\end{equation*}
and $R_{t_m,t_{m+1}}:=\varphi_{t_m,t_{m+1}}-\hat\varphi_{t_m,t_{m+1}}$ 
(see Lord, Malham and Wiese 2008 or Malham and Wiese 2008).
Note that the flow remainder $R_{t_m,t_{m+1}}$ always has the form
\begin{equation*}
R_{t_m,t_{m+1}}=\sum_{|w|\geqslant n+1}\tilde K_w(t_m) V_w,
\end{equation*}
where for the sinh-log series $\tilde K_w=\tfrac12(J_w-J_{\alpha\circ w})$, 
for the exponential Lie series $\tilde K_w=K_{[w]}$ (for each term in the linear combination $V_{[w]}$) 
and for the stochastic Taylor series $\tilde K_w=J_w$. Substituting this
into $\mathcal R$ we see that $\mathcal E^2$ has the form
\begin{equation*}
\mathcal E^2=\sum_{\substack{|u|\geqslant n+1\\|v|\geqslant n+1}}
\biggl(\sum_m\mathcal V_{m,m}(u,v)+\sum_{\ell\neq m}\mathcal V_{\ell,m}(u,v)\biggr),
\end{equation*}
where 
\begin{equation*}
\mathcal V_{\ell,m}(u,v)=\mE\bigl(\hat\varphi_{t_{\ell+1},t_M}\circ \tilde K_u(t_\ell)V_u\circ 
\hat\varphi_{t_0,t_\ell}\circ y_0\bigr)^{\text{\tiny T}}
\bigl(\hat\varphi_{t_{m+1},t_M}\circ \tilde K_v(t_m)V_v\circ \hat\varphi_{t_0,t_m}\circ y_0\bigr).
\end{equation*}
This formula outlines the contribution of the standard accumulation of local errors, 
over successive subintervals of the global interval of integration, to the global error.
Note that to leading order we have
\begin{equation*}
\mathcal V_{m,m}(u,v)=\mE\bigl(\tilde K_u(t_m)\tilde K_v(t_m)\bigr)
\bigl(V_u\circ y_0\bigr)^{\text{\tiny T}}\bigl(V_v\circ y_0\bigr).
\end{equation*}
For the term $\mathcal V_{\ell,m}(u,v)$, we focus for the moment on the case $m<\ell$ 
(our final conclusions below are true irrespective of this). At leading order we have
\begin{align*}
\hat\varphi_{t_0,t_\ell}&=\hat\varphi_{t_{m+1},t_\ell}
\circ\biggl(\id+\sum_{|a|=1}\tilde K_a(t_m)V_a+\cdots\biggr)\circ\hat\varphi_{t_0,t_{m}}\\
&=\id+\sum_{|a|=1}\tilde K_a(t_m)V_a+\cdots,
\intertext{and}
\hat\varphi_{t_{m+1},t_M}&=\hat\varphi_{t_{\ell+1},t_M}
\circ\biggl(\id+\sum_{|b|=1}\tilde K_b(t_\ell)V_b+\cdots\biggr)\circ\hat\varphi_{t_{m+1},t_{\ell}}\\
&=\id+\sum_{|b|=1}\tilde K_b(t_m)V_b+\cdots.
\end{align*}
Substituting these expressions into the form for $\mathcal V_{\ell,m}(u,v)$ above we get
\begin{align*}
\mathcal V_{\ell,m}&(u,v)\\
=&\;\mE\bigl(\tilde K_u(t_\ell)\bigr)\mE\bigl(\tilde K_v(t_m)\bigr)
(V_u\circ y_0)^{\text{\tiny T}}(V_v\circ y_0)\\
&\;+\sum_{|a|=1}\mE\bigl(\tilde K_u(t_\ell)\bigr)\mE\bigl(\tilde K_a(t_m)\tilde K_v(t_m)\bigr)
(V_u\circ V_a\circ y_0)^{\text{\tiny T}}(V_v\circ y_0)\\
&\;+\sum_{|b|=1}\mE\bigl(\tilde K_u(t_\ell)\tilde K_b(t_\ell)\bigr)\mE\bigl(\tilde K_v(t_m)\bigr)
(V_u\circ y_0)^{\text{\tiny T}}(V_b\circ V_v\circ y_0)\\
&\;+\sum_{\substack{|a|=1\\|b|=1}}
\mE\bigl(\tilde K_u(t_\ell)\tilde K_b(t_\ell)\bigr)\mE\bigl(\tilde K_a(t_m)\tilde K_v(t_m)\bigr)
(V_u\circ V_a\circ y_0)^{\text{\tiny T}}(V_b\circ V_v\circ y_0).
\end{align*}
This breakdown allows us to categorize the different mechanisms through which
local errors contribute to the global error at leading order. Indeed in the 
local flow remainder $R$ we distinguish terms with:\smallskip

(1) \emph{zero expectation}: terms $\tilde K_w$ with $|w|=n+1$ of zero expectation
generate terms of order $h^n$ in $\mathcal E^2$ through two routes, through $\mathcal V_{m,m}$ 
and the last term in the expression for $\mathcal V_{\ell,m}(u,v)$ just above.
In $\mathcal V_{m,m}$ they generate order $h^{n+1}$ terms, and the single 
sum over $m$ means that their contribution to the global error $\mathcal E^2$ 
is order $Mh^{n+1}=\mathcal O(h^n)$. In the last term in $\mathcal V_{\ell,m}(u,v)$,
they generate, when the expectations of the products indicated are non-zero,
terms of order $h^{n+2}$; the double sum over $\ell$ and $m$ is then order $h^n$.
Higher order terms $\tilde K_w$ with zero expectation simply generate a higher order
contribution to the global error.

(2) \emph{non-zero expectation}: terms $\tilde K_w$ with $|w|=n+1$ of 
non-zero expectation will generate, through the first term in $\mathcal V_{\ell,m}(u,v)$, 
terms of order $h^{n-1}$---not consistent with an order $n/2$ integrator with global
mean-square error of order $h^n$. 
If any such terms exist in $R$, their expectation must be included 
(which is a cheap additional computational cost) in the integrator, 
i.e.\ we should include $\mE(\tilde K_w)V_w$ in $\hat\psi$.
This will mean that the term left
in $R$ is $\bigl(\tilde K_w-\mE(\tilde K_w)\bigr)V_w$ which has zero
expectation and contributes to the global error through mechanism (1) above.
Further, terms $\tilde K_w$ with $|w|=n+2$ of non-zero expectation will 
generate terms of order $h^n$ in $\mathcal E^2$, i.e.\ they will
also contribute at leading order through the first term in $\mathcal V_{\ell,m}(u,v)$.
\smallskip

The terms of non-zero expectation in $R$, which contribute at leading order to $\mathcal E^2$,
either appear as natural next order terms or through the higher order terms 
$C_2(r\hat\psi+\hat\psi r)$ mentioned in the last section. We will see this explicitly 
in the Simulations section presently. We can, with a cheap additional computational cost,
include them through their expectations in the integrators, so that when we compare
their global errors, the corresponding terms left in the remainders have zero expectation
and are higher order (and thus not involved in the comparison at leading order). 
Further, fortuitously,
the terms of zero expectation which contribute in (1) through the last term in 
$\mathcal V_{\ell,m}(u,v)$ are exactly the same for the stochastic Taylor 
and sinh-log integrators. This is true at all orders and is a result of the
following lemma, and that for the stochastic Taylor and sinh-log expansions
when $|a|=1$, then $\tilde K_a=J_a$.

\begin{lemma}
Suppose $a,w\in\pA$ and $|a|=1$, then we have 
\begin{equation*}
\mE\bigl(J_aJ_w\bigr)\equiv\tfrac12\mE\bigl(J_a(J_w-J_{\alpha\circ w})\bigr).
\end{equation*}
\end{lemma}

\begin{proof} If $|w|$ is even, the expectations on both sides are zero.
If $|w|$ is odd, in the shuffle products $a\sh w$ and $a\sh(\alpha\circ w)$, 
pair off terms where the letter $a$ appears in the same position in an individual term
of $a\sh w$ and the reverse of an individual term of $a\sh(\alpha\circ w)$. The pair,
with the one-half factor, will have the same expectation as the corresponding term in
shuffle product $a\sh w$.
\end{proof}

\subsection{Simulations}
We will demonstrate the properties we proved for the sinh-log series
for numerical integration schemes of strong orders one and three-halves.
We will consider a stochastic differential system with no drift, two
driving Wiener processes and non-commuting governing \emph{linear} vector fields 
$V_i\circ y\equiv a_iy$ for $i=1,2$. 

We focus on the strong order one case first, and extend the analytical 
computations in Lord, Malham and Wiese (2008). With $n=2$, and 
$C_1=1$ and $C_2=-\tfrac12+\epsilon$, 
the mean-square excess $E$, for general $\eps\in\mR$, given by
\begin{equation*}
E=h^3\bigl((U_{112}\,y_0)^{\text{\tiny T}}B(\eps)\,(U_{112}\,y_0)
+(U_{221}\,y_0)^{\text{\tiny T}}B(\eps)\,(U_{221}\,y_0)\bigr).
\end{equation*}
Here $U_{112}=(a_1^2a_2, a_1a_2a_1, a_2a_1^2, a_2^3)^{\text{\tiny T}}$ and  
$U_{221}=(a_2^2a_1, a_2a_1a_2, a_1a_2^2,a_1^3)^{\text{\tiny T}}$ 
are both $4N\times N$ real matrices and the $4N\times4N$
real matrix $B(\eps)$ consists of $N\times N$ blocks of the form
$b(\eps)\otimes I_N$ (here $\otimes$ denotes the Kronecker product) 
where $I_N$ is the $N\times N$ identity matrix and $b(\eps)$ is
\begin{equation*}
\begin{pmatrix}
\tfrac{5}{24}-\bigl(\eps-\tfrac14\bigr)\bigl(3\eps-\tfrac{5}{12}\bigr)&
-\eps\bigl(3\eps-\tfrac{11}{12}\bigr)&
-\bigl(\eps-\tfrac14\bigr)\bigl(3\eps-\tfrac{5}{12}\bigr)&
-3\eps\bigl(3\eps-\tfrac{5}{12}\bigr)\\
-\eps\bigl(3\eps-\tfrac{11}{12}\bigr)&
-\eps\bigl(3\eps-\tfrac23\bigr)&
-\eps\bigl(3\eps-\tfrac{11}{12}\bigr)&
-\eps\bigl(3\eps-\tfrac12\bigr)\\
-\bigl(\eps-\tfrac14\bigr)\bigl(3\eps-\tfrac{5}{12}\bigr)&
-\eps\bigl(3\eps-\tfrac{11}{12}\bigr)&
\tfrac{5}{24}-\bigl(\eps-\tfrac14\bigr)\bigl(3\eps-\tfrac{5}{12}\bigr)&
-3\eps\bigl(3\eps-\tfrac{5}{12}\bigr)\\
-3\eps\bigl(3\eps-\tfrac{5}{12}\bigr)&
-\eps\bigl(3\eps-\tfrac12\bigr)&
-3\eps\bigl(3\eps-\tfrac{5}{12}\bigr)&
-5\eps\bigl(3\eps-1\bigr)
\end{pmatrix}.
\end{equation*}
Each of the eigenvalues of $b(\eps)$ are $N$ multiple eigenvalues for $B(\eps)$.
The eigenvalues of $b(\epsilon)$ are shown in Figure~\ref{eigenvalues}.
The sinh-log expansion corresponds to $\eps=0$, while the  
exponential Lie series corresponds to $\eps=\tfrac16$. 
The eigenvalues of $b(0)$ are $\tfrac{5}{24}$ and zero (thrice)---confirming
our general result for the sinh-log expansion.
However, the eigenvalues of $b(\tfrac16)$ are $\tfrac{5}{24}$, 
$0.5264$, $0.1667$ and $-0.0264$. One negative
eigenvalue means that there are matrices $a_1$ and $a_2$ and initial
conditions $y_0$ for which the order one stochastic Taylor method is
more accurate, in the mean-square sense, than the exponential 
Lie series method (for linear vector fields we also call this the Magnus method).
From Figure~\ref{eigenvalues}, we deduce that for any small values of $\epsilon$ 
away from zero, we cannot guarantee $E>0$ for all possible governing vector fields
and initial data.
The strong order one sinh-log integrator is optimal in this sense. This is
also true at the next order from Corollary~\ref{cor:oddeven}.

\begin{figure}
  \begin{center}
  \includegraphics[width=8cm,height=6cm]{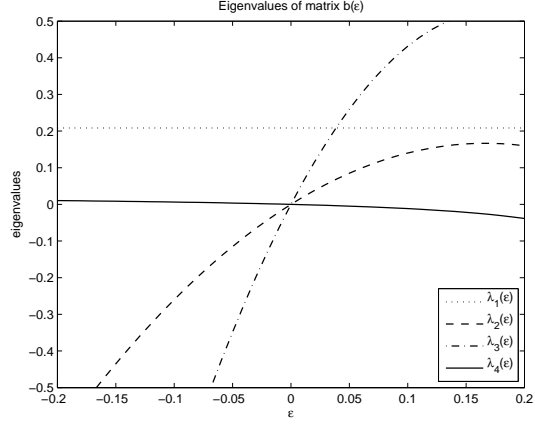}
  \end{center}
  \caption{The eigenvalues, $\lambda_i(\eps)$, $i=1,\ldots,4$ 
of matrix $b(\epsilon)$, as a function of $\epsilon$.}
\label{eigenvalues}
\end{figure}

For our simulations we take $N=2$ and set the coefficient matrices to be
\begin{equation*}
a_1=\begin{pmatrix}
     0.105 &  -7.43\\
     0.03 & 0.345
\end{pmatrix}
\qquad\text{and}\qquad
a_2=\begin{pmatrix}
    -0.065 &  -9.44\\
    -0.005 &   0.265
\end{pmatrix}.
\end{equation*}
In Figure~\ref{contours}, using these matrices, we plot the mean-square
excess $E^{\mathrm{ls}}$ for the exponential Lie series and 
$E^{\mathrm{sl}}$ for the sinh-log series, as a function of the two components 
of $y_0$. We see there are regions of the phase space where $E^{\mathrm{ls}}$
is negative---of course $E^{\mathrm{sl}}$ is positive everywhere.
Hence if the solution $y_t$ of the stochastic differential
system governed by the vectors fields $V_i\circ y=a_i y$, $i=1,2$,
remains in the region where $E^{\mathrm{ls}}$ is negative, then
the numerical scheme based on the order one exponential Lie series
is less accurate than the stochastic Taylor method. Note that for
the stochastic Taylor method, we need to include the terms 
\begin{equation*}
\tfrac18 h^2(a_1^4+a_1^2a_2^2+a_2^2a_1^2+a_2^4)
\end{equation*}
in the integrator. These are the expectation of terms with 
$|w|=4$ which contribute at leading order in the global
error (only), and which can be cheaply included in the 
stochastic Taylor integrator. For the  
exponential Lie series we include the terms
\begin{equation*}
\tfrac{1}{24} h^2([a_2,[a_2,a_1]]a_1+[a_1,[a_1,a_2]]a_2+a_2[a_1,[a_1,a_2]]+a_1[a_2,[a_2,a_1]]),
\end{equation*}
in $\hat\psi^{\mathrm{ls}}$. These are non-zero expectation terms with $|w|=4$ 
which contribute at leading order in the global
error through $-C_2(r\hat\psi+\hat\psi r)$, where $C_2=-\tfrac12$.
In the same vein, 
for the sinh-log integrator, we include in $\hat\psi^{\mathrm{sl}}$ the terms
\begin{equation*}
\tfrac14 h^2(2a_1^4+a_2^2a_1^2+a_2a_1^2a_2+a_1a_2^2a_1+a_1^2a_2^2+2a_2^4).
\end{equation*}
Figure~\ref{errorplot1p0} shows the global error versus
time for all three integrators for the linear system.
We used the global interval of integration
$[0,0.0002]$, starting with $y_0=(19.198, 28.972)^{\mathrm{T}}$,
and stepsize $h=2.5\times 10^{-5}$.
With this initial data, the small global interval of integration means that
all ten thousand paths simulated stayed within the region of the 
phase space where $E^{\mathrm{ls}}$ is negative in Figure~\ref{contours}. 

\begin{figure}
  \begin{center}
  \includegraphics[width=6cm,height=6cm]{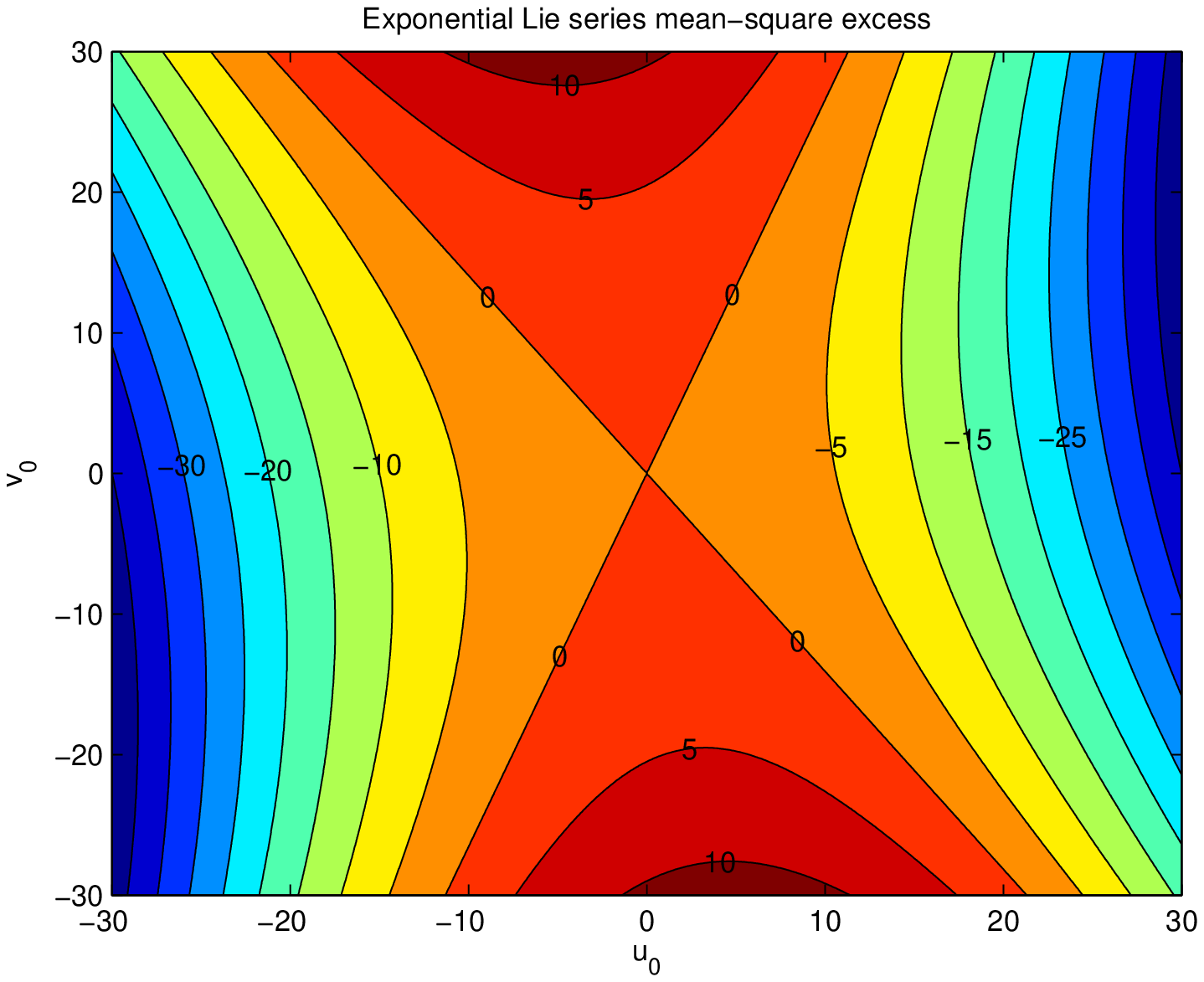}
  \includegraphics[width=6cm,height=6cm]{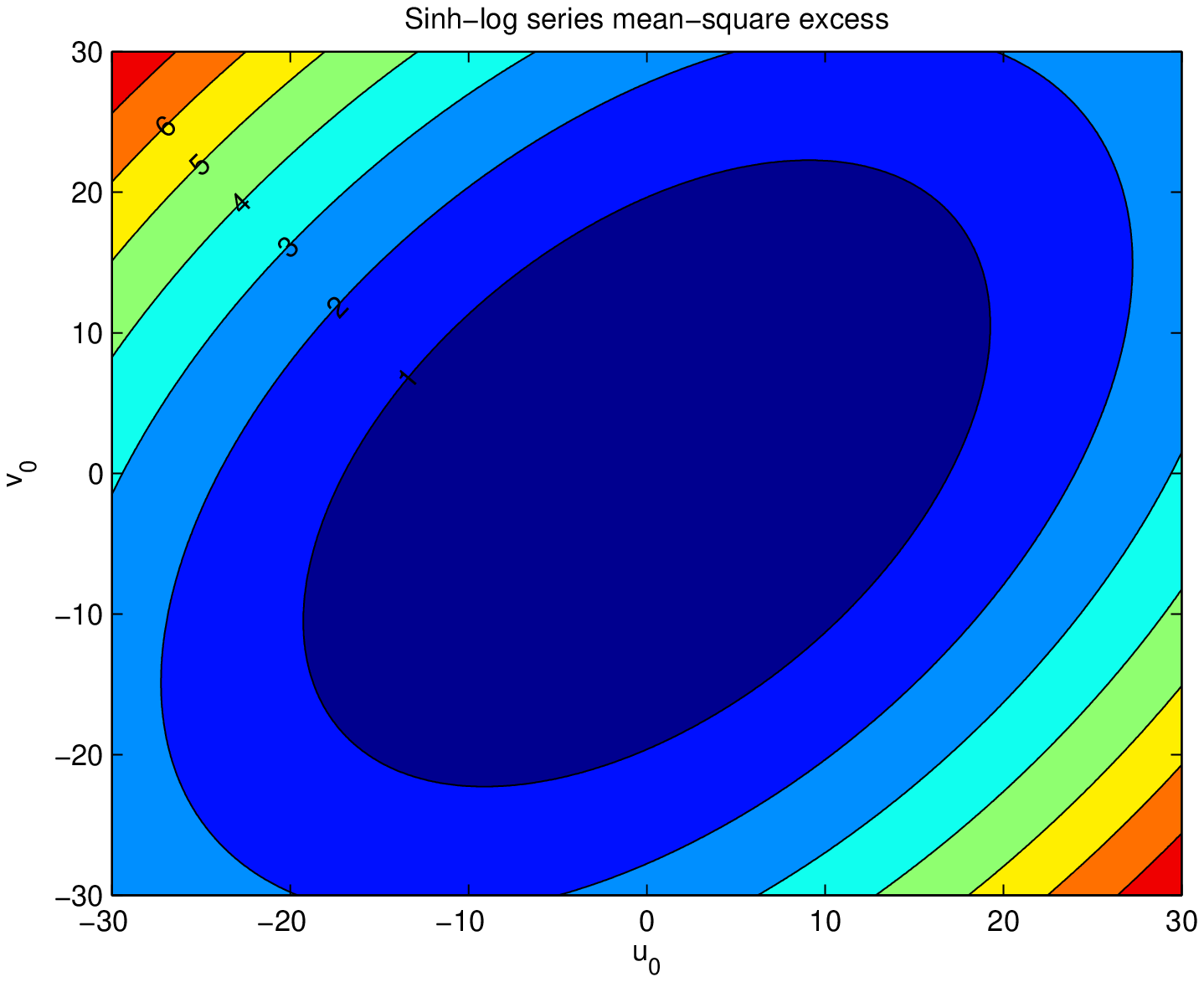}
  \end{center}
  \caption{Contour plots of the mean-square excess as a function 
of the two components of the data $y_0=(u_0,v_0)^{\text{\tiny T}}$,
for the strong order one example, for the exponential Lie series 
(left panel) and the sinh-log series (right panel).}  
\label{contours}
\end{figure}

\begin{figure}
  \begin{center}
  \includegraphics[width=7cm,height=5cm]{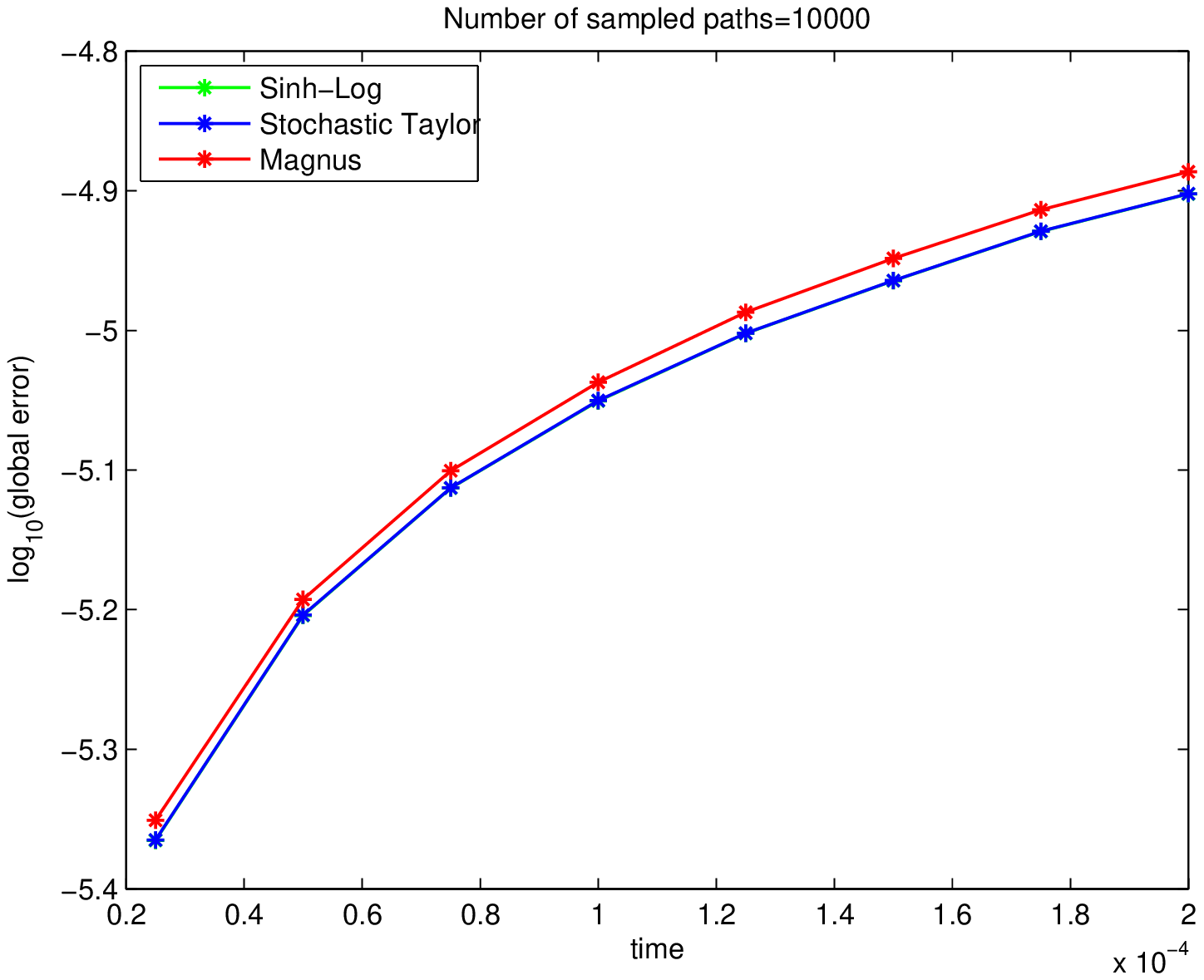}\\
  \includegraphics[width=7cm,height=5cm]{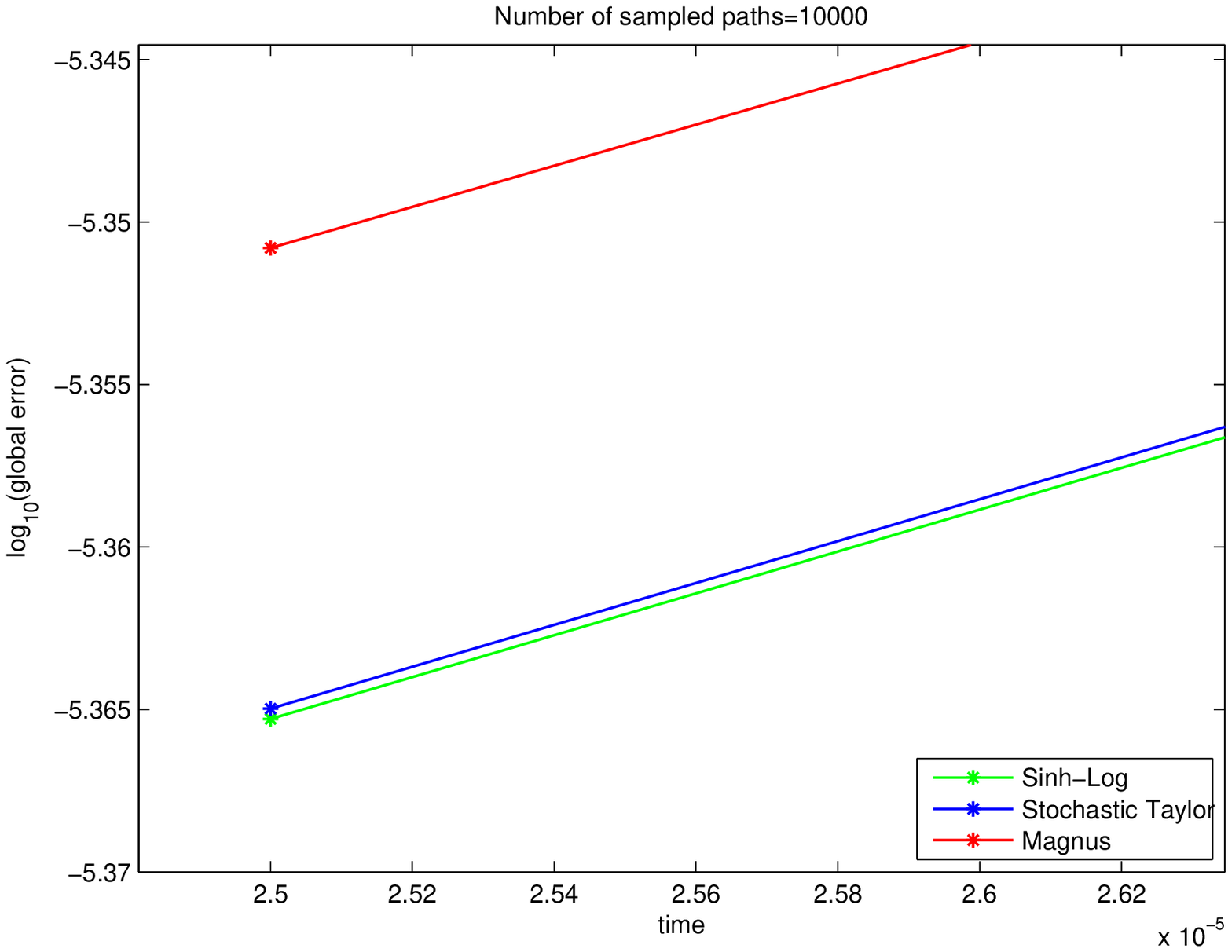}\\
  \includegraphics[width=7cm,height=5cm]{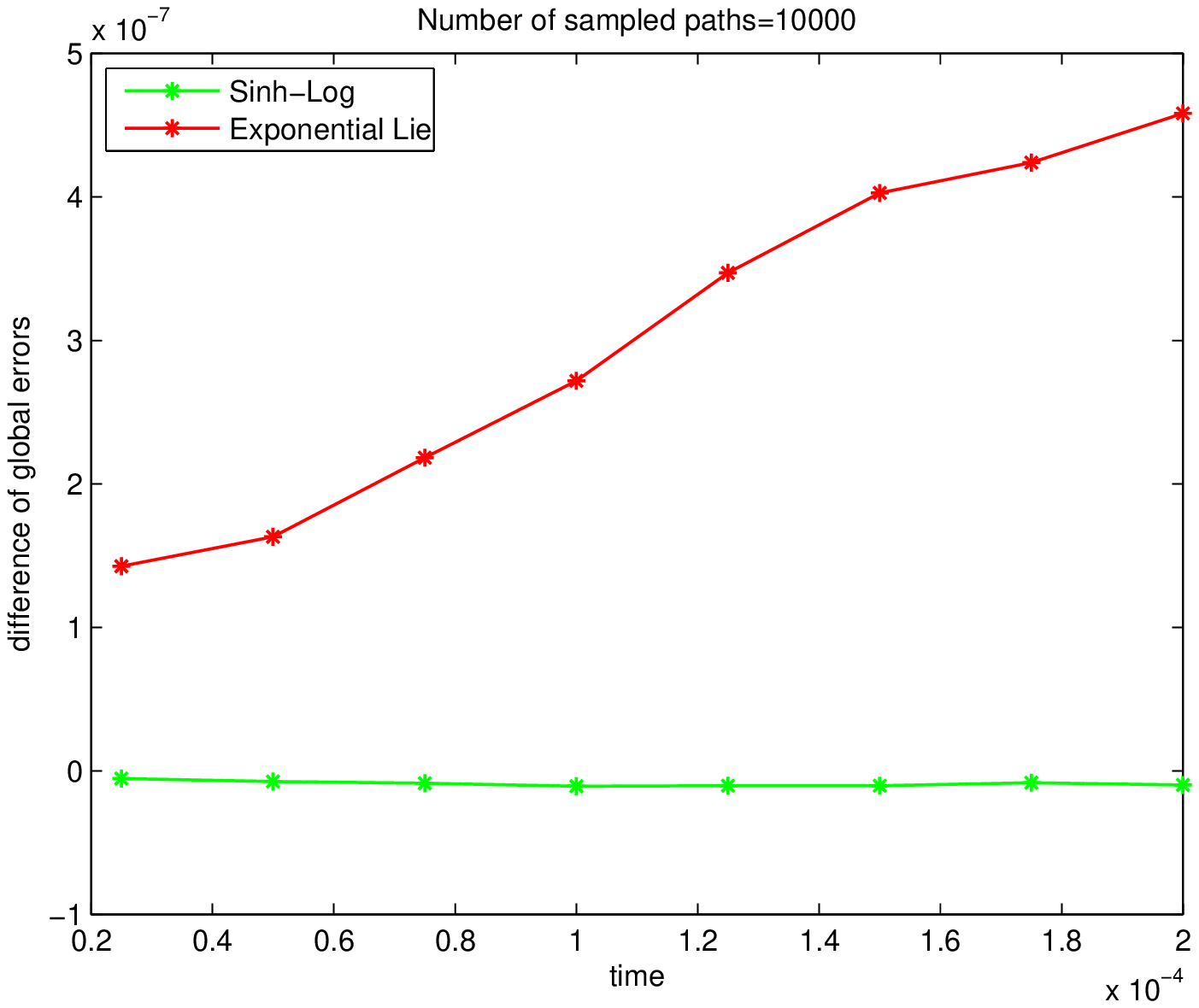} 
  \end{center}
  \caption{Mean-square global error vs time plot for the sinh-log, 
exponential Lie (Magnus) and stochastic Taylor methods for the order 
one example. The top panel shows the error, and the middle
panel a magnification of the left region of the plot in the 
top panel. The lower panel shows the differences between, the 
global sinh-log and exponential Lie errors, and the error of
stochastic Taylor method.}  
\label{errorplot1p0}
\end{figure}

The error for the exponential 
Lie series integrator, we see in Figure~\ref{errorplot1p0}, 
is larger than that for the stochastic Taylor integrator. The
error for the sinh-log integrator is smaller, though only
marginally so. In fact it is hardly discernible from the stochastic
Taylor plot, so the middle panel shows a magnification 
of the left region of the plot in the top panel. We
plot the differences between the errors in the 
lower panel to confirm the better performance of the
sinh-log integrator over the global interval. Further, 
estimates for the local errors for the sinh-log and Lie series 
integrators from the data in Figure~\ref{errorplot1p0}, of course, 
quantitatively match analytical estimates for the mean-square 
excess $E$ above. 

\begin{remark}
Generically the Castell--Gaines method of strong order
one markedly outperforms the sinh-log method (which itself 
outperforms the stochastic Taylor method more markedly). However,
as we have seen, there are pathological cases for which
this is not true.
\end{remark}

In Figure~\ref{errorplot1p5} we compare the global errors for the
strong order three-halves sinh-log and stochastic Taylor methods,
with governing linear vector fields with coefficient matrices
\begin{equation*}
a_1=\begin{pmatrix}0 & 1\\ -1/2 & -51/200         
\end{pmatrix}
\qquad\text{and}\qquad
a_2=\begin{pmatrix}1 & 1\\ 1 & 1/2
\end{pmatrix}.
\end{equation*}
and initial data $y_0=(1, 1/2)^{\text{\tiny T}}$. Again as expected we see that
the stochastic Taylor method is less accurate than the sinh-log method
for sufficienltly small stepsizes.

\begin{figure}
  \begin{center}
  \includegraphics[width=7cm,height=5cm]{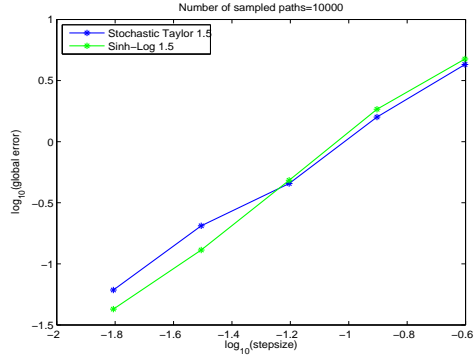}
  \end{center}
  \caption{Global error vs stepsize plot for the sinh-log 
and stochastic Taylor methods of strong order three-halves example.}  
\label{errorplot1p5}
\end{figure}

\begin{remark}
There is one caveat we have not mentioned thusfar. Constructing the 
approximation $\hat\varphi_t$ from $\hat\psi_t$ is in general nontrivial.
For linear vector fields $V_i\circ y=a_i\,y$, we know that $\hat\psi_t$
is simply a matrix, and we can straightforwardly construct $\hat\varphi_t$
using the matrix square root. For general nonlinear vector fields,
we have not as yet found a superior method to simply expanding the square root
shown to a sufficient number of high degree terms.
\end{remark}

\section{Concluding remarks}
\label{sec:conclu}
We have shown that the mean-square remainder associated with the sinh-log
series is always smaller than the corresponding stochastic Taylor mean-square
remainder, when there is no drift, to all orders. Since the order one-half
sinh-log numerical method is the same as the order one-half Castell--Gaines
method, it trivially inherits the asymptotically efficient property as well
(indeed, if we include a drift term as well). We have not endeavoured to
prove asymptotic efficiency more generally.
However, in Section~\ref{sec:prac} we demonstrated that for two driving
Wiener processes, the order one sinh-log numerical method is optimal
in the following sense. From Figure~\ref{eigenvalues}, we see that
any deviation of $\eps$ from zero will generate a negative eigenvalue 
for $b(\eps)$. Consequently there exist vector fields such that the mean-square excess 
will be negative in regions of the phase space. Further, 
from Corollary~\ref{cor:oddeven}, the order three-halves sinh-log integrator
is also optimal in this sense.
These results are only true when there is no drift,
and it could be argued that our simulations
and demonstrations are somewhat academic. However the results we
proved have application in splitting methods 
for stochastic differential equations. For example, in stochastic volatility 
simulation in finance, Ninomiya \& Victoir (2006) and 
Halley, Malham \& Wiese (2008) simulate the Heston model for 
financial derivative pricing, and use splitting to preserve positivity for
the volatility numerically. They employ a Strang splitting that separates
the diffusion flow from the drift flow and requires a distinct simulation
of the purely diffusion governed flow. 

Why is the sinh-log expansion the answer? This result is intimately tied to the 
mean-square error measure we chose. The terms in the remainder of
any truncation contain multiple Stratonovich integrals. Associated
with each one is a mean-square error. There is a structure underlying
these expectations.
The sinh-log expansion somehow encodes this structure in an optimal
form, it emulates the stochastic Taylor information more concisely.
The next question is of course, what is the best structure 
when we include drift? Answering this is our next plan of action. 


\begin{acknowledgements}
We would like to thank Peter Friz, Terry Lyons and Hans Munthe--Kaas 
for interesting and useful discussions related to this work.
We would also like to thank the anonymous referees for their 
critique and suggestions which helped improve the original manuscript.
\end{acknowledgements}

\label{lastpage}

\end{document}